\documentclass[11pt]{amsart}

\usepackage[mathlines]{lineno}

\usepackage{amsmath, amscd, amsthm, amssymb, graphics, mathrsfs, setspace,fancyhdr,geometry,tabularx,shapepar, hyperref, xcolor, tikz, cite, booktabs, dsfont,amsfonts, marvosym, amsbsy, bm, tikz, array, mathtools,arydshln,tikz-cd}
\usetikzlibrary{matrix,arrows,backgrounds}
\usepackage[all]{xy}
\usepackage{leftidx}
\usepackage[all]{xy}
\usepackage{verbatim}
\usepackage{soul}
\usepackage{subcaption}
\usepackage{graphicx}
\usepackage{multicol}
\usepackage{xcolor}

\definecolor{apple}{rgb}{0.30,0.75, 0.05}

\usepackage[shortlabels]{enumitem}
\usepackage{tikz-cd}
\usepackage{pinlabel}

\usepackage{mathrsfs}
\usepackage{wrapfig}
\usepackage{listings}

\usepackage[OT2,T1]{fontenc}
\DeclareSymbolFont{cyrletters}{OT2}{wncyr}{m}{n}
\DeclareMathSymbol{\Sha}{\mathalpha}{cyrletters}{"58}
\DeclareMathSymbol{\Zhe}{\mathalpha}{cyrletters}{17}
\usepackage{wasysym}

\usepackage{apptools}
\usepackage{chngcntr}

\mathchardef\hyph="2D

\hypersetup{colorlinks=false}
\theoremstyle{plain}
\newtheorem{theorem}{Theorem}[section]

\newtheorem{lemma}[theorem]{Lemma}

\newtheorem{proposition}[theorem]{Proposition}

\newtheorem{question}[theorem]{Question}

\newtheorem*{theorem*}{Theorem}
\newtheorem*{problem*}{Problem}
\newtheorem*{question*}{Question}

\theoremstyle{definition}
\newtheorem{definition}[theorem]{Definition}

\newtheorem{remark}[theorem]{Remark}

\newtheorem{example}[theorem]{Example}

\AtAppendix{\counterwithin{theorem}{subsection}}

\numberwithin{equation}{section}

\newcommand{\ob}{\mathcal{OB}}

\newcommand{\leftexp}[2]{{\vphantom{#2}}^{#1}{#2}}
\newcommand{\weezer}{\leftexp{=}{\kern-0.23em\mathsf{W}}^{\kern-0.21em =}}

\begin{document}

\title{Relative group trisections}

\author[Castro]{Nickolas Andres Castro}
\address{Castro, Department of Mathematics, Rice University,
Houston, TX 77005}
\email{ncastro.math@rice.edu}
\urladdr{\url{https://nickcastromath.com/}}

\author[Joseph]{Jason Joseph}
\address{Joseph, North Carolina School of Science and Mathematics,
Morganton, NC 28655}
\email{jason.joseph@ncssm.edu}
\urladdr{\url{https://jmjoseph22.wordpress.com/}}

\author[McFaddin]{Patrick K. McFaddin}
\address{McFaddin, Department of Mathematics, Fordham University, 
New York, NY 10023}
\email{pkmcfaddin@gmail.com}
\urladdr{\url{http://mcfaddin.github.io/}}

\begin{abstract}
Trisections of closed 4-manifolds, first defined and studied by Gay and Kirby, have proved to be a useful tool in the systematic analysis of 4-manifolds via handlebodies. Subsequent work of Abrams, Gay, and Kirby established a connection with the algebraic notion of a group trisection, which strikingly defines a one-to-one correspondence. We generalize the notion of a group trisection to the non-closed case by defining and studying relative group trisections. We establish an analogous one-to-one correspondence between relative trisections and relative group trisections up to equivalence. The key lemma in the construction may be of independent interest, as it generalizes the classical fact that there is a unique handlebody extension of a surface realizing a given surjection. Moreover, we establish a functorial relationship between relative trisections of manifolds and groups, extending work of Klug in the closed case.
\end{abstract}

\maketitle

\section{Introduction}
\label{section:introduction}
In this paper we investigate an algebraic structure associated to a smooth trisected $4$--manifold relative to its non-empty, connected boundary, called a \emph{relative group trisection}. This is a natural extension of the group-theoretic framework established by Abrams, Gay, and Kirby in \cite{grouptri}, which includes smooth manifolds with boundary. In both the closed and compact-with-boundary setting, a trisected smooth $4$--manifold $X$ is decomposed into three diffeomorphic $4$--dimensional $1$-handlebodies. The differences and similarities lie in how the pieces intersect. In the closed setting, the triple intersection is a closed surface, and the pairwise intersections are $3$--dimensional handlebodies, which can be thought of as cobordism from a closed surface to the empty set equipped with a Morse function with only index 2 critical points. In the relative setting (i.e., manifolds with boundary), the triple intersection is a surface with boundary, and the pairwise intersections are specific types of compression bodies (defined in Section~\ref{subsection:rel_trisections}), which may be viewed as cobordisms from a compact surface with boundary to another surface with boundary equipped with a Morse function with only index 2 critical points. By taking the fundamental group of each of the pieces, along with their double and triple intersections, we obtain the structure of a trisected, finitely presented group $G = \pi_1(X^4)$. Presented herein is an inverse to the fundamental group map $\pi_1$, viewed as a function on the set of diffeomorphism classes of relatively trisected 4--manifolds with boundary. This shows that every finitely presented group can be given the structure of a (relative) trisection.

\begin{theorem*}[see \ref{thm:set_equiv_relative}]
There is a map $\mathcal{M}$ from the set of relative group trisections up to isomorphism to the set of relatively trisected manifolds up to diffeomorphism. Moreover, $\pi_1 \circ\mathcal{M}$ is the identity up to isomorphism of relatively trisected groups, and $\mathcal{M}\circ \pi_1$ is the identity up to diffeomorphism of relatively trisected $4$--manifolds.
\end{theorem*}

In principle, a trisection encapsulates all of the smooth data of a smooth $4$--manifold. Thus, group trisections provide an alternative algebraic approach to investigating smooth $4$--manifolds and trisections. This is seen in \cite{grouptri}, where the authors establish that the Smooth $4$--dimensional Poincar\'{e} Conjecture is equivalent to the statement:  ‘‘Every $(3k,k)$--group trisection of the trivial group is equivalent to the $(0,0)$--trisection of the trivial group.''
Using this translation of one of the most fundamental open questions in low-dimensional topology, one may view relative group trisections and Theorem~\ref{thm:equiv_relative} as a first step towards a similar group theoretic understanding of smooth structures on the $4$--ball and the Smooth $4$--dimensional Sch\"{o}enflies Conjecture.

Using the functor $\pi_1$ to obtain a purely algebraic framework has also been implemented in the setting of knotted surfaces in closed $4$--manifolds via bridge trisections in \cite{grouptrisurf}. One of the main results of loc. cit.  is the bijection between the \emph{set} $\mathtt{Man}^{(4,2)}/\sim$ of bridge trisections of surfaces in $4$--manifolds up to the appropriate notion of equivalence, and the \emph{set} $\mathtt{Alg}^{(4,2)}/\sim$ of triples of maps called \emph{bounding homomorphisms}, again up to the appropriate notion of equivalence. By applying $\pi_1$ to the exterior of a surface $S$ in bridge position with respect to a closed trisection $X^4$, the authors obtain a commutative cube which satisfies the conditions of a relative group trisection corresponding to a relative trisection of the exterior of $S$. This exhibits how relative group trisections agree with and extend the existing algebraic framework of group trisections.

We continue to flesh out the theory by providing a categorical enrichment of the above theorem, extending work of Klug \cite{klugtri}. Let $\mathsf{Tris4Man}^{\partial}$ denote the category of relatively trisected 4-manifolds with boundary and smooth maps preserving the relative trisections; let $\mathsf{3SplitHom}$ denote the category of relative 3-splitting homomorphisms (see Subsection~\ref{subsection:cat_4man}). The data of a relative 3-splitting homomorphism is equivalent to a relative group trisection.

\begin{theorem*}[see \ref{thm:equiv_relative}]
There is a functor $\mathscr M_4^{\partial}:  \mathsf{3SplHom}^{\partial} \to \mathsf{Tris4Man}^{\partial}$ which agrees with the above map $\mathcal M$ on objects and which defines a categorical equivalence $ \mathsf{3SplHom}^{\partial}\simeq \mathsf{Tris4Man}^{\partial}$.
\end{theorem*}

This establishes a purely algebraic formulation of the notion of a relative trisection, providing a natural context in which to view these topological decompositions.

\subsection*{Organization} In Section~\ref{section:trisections}, we recall the notion of trisections of closed 4-manifolds and groups, and restate the main theorem of \cite{grouptri}. In Section~\ref{section:relative_case}, we treat the case of manifolds with boundary and introduce \emph{relative} trisections of 4-manifolds and groups. In Section~\ref{section:category_trisection}, we discuss the categorical enrichments of Heegaard splittings and trisections developed in \cite{klugtri}. We extend the work of Klug to the relative case, establishing a categorical equivalence between relative trisections of 4-manifolds and groups. We then put forward natural questions and directions for future work in Section~\ref{section:remarks}.

\subsection*{Conventions and Notation}
All manifolds are smooth, compact, connected, and oriented. For a topological space $X$ and a point $x \in X$, we let $\pi_1(X,x)$ denote the fundamental group of homotopy classes of loops in $X$ based at $x$. We often suppress mention of the basepoint, denoting the fundamental group by $\pi_1(X)$. We use $\pi_1$ when we wish to emphasize the fundamental group as a functor on the category of (pointed) topological spaces and (pointed) continuous maps.

\subsection*{Acknowledgements}
The authors would like to thank David Gay and Hannah Schwartz for helpful conversations. The second author would also like to thank Benjamin Ruppik. The second author was partially supported by NSF grants DMS-1664567 and DMS-1745670. Via the third author, this work was supported by a Spring 2023 Fordham Faculty Fellowship and a 2023-2024 Fordham Faculty Research Grant. 
\section{Trisections}
\label{section:trisections}
 
We recall the basic definitions of closed trisected $4$--manifolds \cite{gk}, and relate them to group trisections \cite{grouptri}.
\begin{definition}
Given non-negative integers $g$ and $k$, a \emph{$(g,k)$-trisection} of a closed, oriented, smooth $4$--manifold $X$ is a decomposition $X=X_1\cup X_2 \cup X_3$ such that 
\begin{enumerate}[i)]
	\item $X_i \cong \natural^k S^1\times B^3$,
	\item $\partial X_i = (X_i\cap X_{i-1})\cup (X_i \cap X_{i+1})$ is a genus $g$ Heegaard splitting, where the indices are taken modulo $3$,
	\item $X_1\cap X_2\cap X_3 = \Sigma_g,$ a closed genus $g$ surface.
\end{enumerate}
From this data we obtain that the Euler characteristic of $X$ is $\chi(X)=2+g-3k.$ Thus, specifying the $4$--manifold and the genus of the trisection determines $k$.
\end{definition}

There is a stabilization operation for closed trisected $4$--manifolds whereby we take the connected sum of a given trisection with the genus $3$ trisection of $S^4$. This operation extends to trisected manifolds with boundary. As this operation is performed in the interior of the manifold, we refer to it as an \emph{interior stabilization}.

\begin{theorem}[\hspace{1sp}\cite{gk}]
Every smooth, oriented, closed $4$--manifold admits a $(g,k)$-trisection for some $g$ and $k$. Moreover, any two trisections of a fixed $4$-manifold can be made isotopic after some number of interior stabilizations of each.
\end{theorem}

Given non-negative integers $g$ and $k$, let us fix the following group presentations:
\begin{itemize}
	\item $\displaystyle{S_g = \left\langle x_1, y_1, \ldots, x_g,y_g \; | \; \prod_{i=1}^g [x_i, y_i] = 1 \right\rangle}$,
	\item $\displaystyle{H_g =\left\langle \delta_1,\ldots,\delta_g\right\rangle}$ is the free group of rank $g$, 
	\item $Z_k=\left\langle z_1, \ldots, z_k \right\rangle$ is the free group of rank $k$.
\end{itemize}

\begin{definition}\label{defn:group_trisection}
Let $G$ be a finitely presented group. A $(g,k)$-\emph{trisection} of $G$ is a commutative cube of group homomorphisms
\begin{center}
\begin{tikzcd}[row sep=scriptsize, column sep=scriptsize]
    & H_g \arrow[r] \arrow[dr] &   Z_k \arrow[dr]  &\\
    S_g \arrow[r]\arrow[dr]\arrow[ur]  &   H_g \arrow[ur, crossing over] \arrow[dr]   &   Z_k \arrow[r]  &   G\\
    &   H_g \arrow[r] \arrow[ur] &   Z_k \arrow[ur] \arrow[ul, leftarrow, crossing over] &
\end{tikzcd}
\end{center}
where each map is surjective and each face is a push-out.
\end{definition}

Given a $(g,k)$-trisection of a closed $4$--manifold $X=X_1 \cup X_2\cup X_3$, applying the fundamental group to each $X_i$, their pairwise intersections, and triple intersection yields a $(g,k)$-trisection of $\pi_1(X)$. Thus, the functoriality of $\pi_1$ defines a map from the set of pointed, parametrized, closed, trisected $4$--manifolds to the set of trisected groups. Perhaps surprisingly, this is a one-to-one correspondence. 

\begin{theorem}[\hspace{1sp}\cite{grouptri}]\label{thm:man2grp}
There is a map $\mathcal{M}$ from the set of group trisections to the set of pointed, parametrized, closed, trisected $4$--manifolds. Moreover, $\pi_1 \circ\mathcal{M}$ is the identity up to isomorphism of trisected groups and $\mathcal{M}\circ \pi_1$ is the identity up to diffeomorphism of trisected $4$--manifolds.
\end{theorem}

Subsequent work of Klug establishes an equivalence of categories which induces the above bijection. We revisit this work in Section~\ref{subsection:cat_4man}.

\section{Relative group trisections and $4$--manifolds with boundary}
\label{section:relative_case}

We consider trisections of compact $4$--manifolds relative to a connected boundary. The main result of this section is Theorem~\ref{thm:set_equiv_relative}, which relates relative trisections of 4-manifolds with boundary and relative group trisections. This generalizes Theorem~\ref{thm:man2grp} to the relative case. We finish the section by recalling results of Klug which will be useful in establishing a categorical enrichment of our main result.

\subsection{Relative trisections}
\label{subsection:rel_trisections}

As mentioned in the introduction, results in this section extend to all compact $4$--manifolds, but we restrict our attention to manifolds with non-empty, connected boundary for the sake of simplicity. The codimension $0$ pieces of a trisected manifold, with or without boundary, are $4$--dimensional $1$--handlebodies. The key distinction then is in how the handlebodies intersect. 

We begin by recalling the structure of the pairwise intersections which are specific types of relative compression bodies, a schematic of which is shown in Figure~\ref{fig:CSchematic}. First, obtain a genus $g$ surface with boundary $\Sigma_g^b$ from $S_g^b$ which is unique up to diffeomorphism. We construct a compression body $W_{g,p}^b$ by attaching $3$--dimensional $2$--handles to $\Sigma_g^b \times I$ along curves in $\Sigma_g^b \times \{1\}$, so that $W_{g,p}^b$ is a cobordism with sides from $\Sigma_g^b$ to a connected genus $p$ surface with $b$ boundary components, $P$. We will decompose $\partial W_{g,p}^b = \partial^-W \cup \partial^0W\cup \partial ^+W$, where $\partial^-W \cong \Sigma_g^b$, $\partial^+W \cong P$, and $\partial^0W\cong \partial\Sigma_g^b \times I$. However, we do exclude the setting when $p>b=0$.

\begin{definition}\label{defn:relative_trisection}
Let $g,k,p, b$ be non-negative integers with $g\geq k \geq0$, $p\geq0$, and $b>0$. A \emph{$(g,k;p,b)$-trisection} of a $4$-manifold $X$ with non-empty, connected boundary is a decomposition, $X = X_1 \cup X_2\cup X_3$, such that
	\begin{enumerate}[i)]
		\item $X_i \cong \natural ^k S^1\times B^3$
		\item $X_i \cap X_{i\pm 1} \cong W_{g,p}^b$
		\item $X_1\cap X_2 \cap X_3 = \Sigma_g^b$ is a genus $g$ surface with $b$ boundary components.
	\end{enumerate}
Implicit in this definition is $X_i \cap \partial X\cong P$ for each $i$. We also refer to such decompositions simply as \emph{relative trisections} if $g,k,p,b$ do not need to be explicitly identified. 
\end{definition}

\begin{remark}[Empty Boundary]\label{rem:empty}
It is possible to allow $b = 0$ above, in which case we must have $p = 0$. We may extend Definition~\ref{defn:relative_trisection} to include the case when $p=b=0$, which yields a trisected $4$--manifold relative to an \emph{empty boundary}. We then identify a $(g,k;0,0)$--trisection with a $(g,k)$--trisection. Indeed, a compression body $W_{g,0}^0$ is a cobordism from a closed genus $g$ surface to a closed genus $0$ surface. Filling in the sphere boundary component yields the standard construction of a genus $g$ handlebody, each pair of which gives a Heegaard decomposition of one of the $\partial X_i$. 
\end{remark}

{\centering
\begin{figure}
	\labellist                             
	    \small\hair 2pt
		\pinlabel $\color{red}\vdots$ at 110 112
		\pinlabel $\Sigma_g^b$ at 110 15
		\pinlabel $P$ at 110 195
		\pinlabel \resizebox{6pt}{.9in}{\mbox{$\{$}} at 20 108
		\pinlabel \resizebox{6pt}{.9in}{\mbox{$\}$}} at 198 108
		\pinlabel \mbox{$(\partial\Sigma_g^b \times I)\cong \partial^0 W$} at -50 108
		\pinlabel \mbox{$\partial^0 W\cong (\partial\Sigma_g^b \times I)$} at 280 108
	\endlabellist
	\includegraphics[scale=.6]{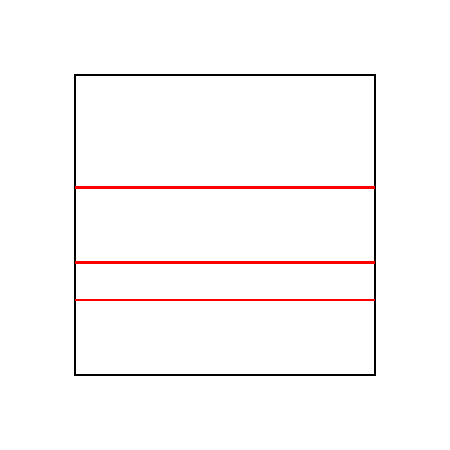}
\caption[width=textwidth]{Pictured is a schematic for the compression body $W$ viewed as a cobordism from $\Sigma_g^b$ (bottom) to $P$,a genus $p$ surface with $b$ boundary components (top). The red horizontal lines indicate critical levels (corresponding to index $2$ critical points) of a Morse function $f: W \to [0,1]$, where $f^{-1}(0)=\Sigma_g^b$ and $f^{-1}(1)= P.$ The vertical sides of the schematic comprise $\partial^0 W\cong \partial\Sigma_g^b \times I.$\label{fig:CSchematic}}
\end{figure}
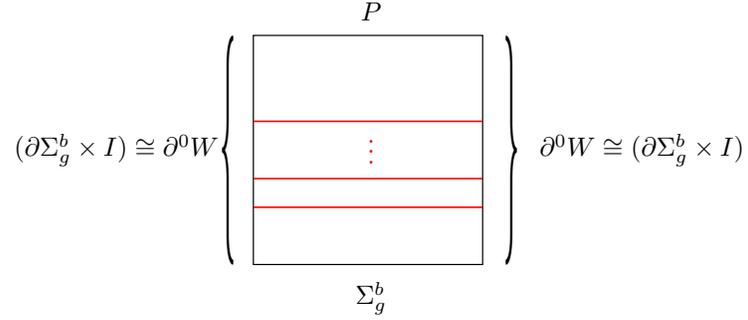}

Given non-negative integers $g,k,p,b$ as in Def.~\ref{defn:relative_trisection}, let us fix the following group presentations:
\begin{itemize}
	\item $\displaystyle{S_g^b = \left\langle x_1, y_1, \dots, x_g,y_g, w_1,\dots,w_b \; | \; \prod_{i=1}^g [x_i, y_i] = \prod_{j=1}^b w_j \right\rangle}$,
	\item $\displaystyle{C_{g,p}^b=\left\langle \delta_1,\dots,\delta_{g-p}, \zeta_1,\eta_1,\dots,\zeta_p,\eta_p, \omega_1,\dots,\omega_b \; | \; \prod_{i=1}^p [\zeta_i,\eta_i]= \prod_{j=1}^b \omega_i \right\rangle}$ is the free group of rank $n=g+p + b-1$, \\
	\item $Z_k$ is the free group of rank $k$.
\end{itemize}

\begin{definition}\label{def:relgroup}
Let $G$ be a finitely presented group. A \emph{$(g,k;p,b)$-trisection} of $G$ is a commutative cube of group homomorphisms
\begin{center}
\begin{tikzcd}[row sep=scriptsize, column sep=scriptsize]
    & C_{g,p}^b \arrow[r] \arrow[dr] &   Z_k \arrow[dr]  &\\
    S_g^b \arrow[r]\arrow[dr]\arrow[ur]  &   C_{g,p}^b \arrow[ur, crossing over] \arrow[dr]   &   Z_k \arrow[r]  &   G\\
    &   C_{g,p}^b \arrow[r] \arrow[ur] &   Z_k \arrow[ur] \arrow[ul, leftarrow, crossing over] &
\end{tikzcd}
\end{center}
such that
\begin{enumerate}[i)]
	\item each map is surjective,
	\item each face of the cube is a pushout,
	\item for each map $S_g^b \to C_{g,p}^b$ and each $j\in\{1,\ldots, b\}$, we have $w_j \mapsto \omega_j$.

\end{enumerate}
This will be referred to as a \emph{relative group trisection} of $G$.
\end{definition}

\begin{remark}[Relation to group trisections]\label{rem:grouptri}
Just as in the manifold setting, we can recover the definition of a $(g,k)$-group trisection defined in \cite{grouptri} by allowing for $(g,k;0,0)$--group trisections in Definition~\ref{def:relgroup}.  In this case, $S_{g,0}^0$ does not contain the elements $w_i$ in the generating set and the relation becomes $\prod[x_i,y_i]=1$. That is, $S_{g,0}^0$ is a closed surface group. Additionally, $C_{g,0}^0$ is a free group of rank $g$. In this case, the third condition in the definition is vacuous, and we recover the notion of a $(g,k)$--group trisection.
\end{remark}

Given a $(g,k;p,b)$-trisection of $G$, we construct a trisected $4$--manifold $X$ with boundary unique up to diffeomorphism with $\pi_1(X)=G.$ Our argument is similar to that given in \cite{grouptri}, with key differences in constructing the compression bodies $X_i\cap X_j$ and in the construction of $\partial X$. We have the following lemma, which generalizes the classical fact that there is a unique handlebody inclusion map realizing any surjection from a closed surface group to a free group of half the rank. This proof follows the argument given in \cite{LR}, but more care is required to ensure that the correct diffeomorphism type of compression body is created.

\begin{lemma}\label{lem:compression}
Given a surjection $h:S_g^b \rightarrow C_{g,p}^b$ as in Definition~\ref{def:relgroup}, there is a compression body $W=W_{g,p}^b$, unique up to diffeomorphism relative to $\partial^-W = \Sigma_g^b=\Sigma$, that induces $h$ on the fundamental group.
\end{lemma}

\begin{proof}

Let $X_n=\vee^nS^1$ where $n=g+p+b-1$ and identify the oriented generators of $\pi_1(X_n)$ with all but one of the elements in our generating set for $C_{g,p}^b$. Since we can write any $\omega_j$ in terms of the other generators, we make the convention to exclude $\omega_b$ from our wedge of circles. This gives an isomorphism $\pi_1(X_n)\to C_{g,p}^b$. We also fix an isomorphism $\phi: S_g^b \to \pi_1(\Sigma)$ that maps the elements $w_1,\dots,w_b$ of $S_g^b$ onto fundamental group elements corresponding to the $b$ boundary components of $\Sigma$. Let $f:\Sigma \rightarrow X_n$ be a continuous map that induces $h$ on $\pi_1$, i.e.\ so that $f_*\circ \phi=h$.

For each circle corresponding to a $\delta_i$ generator of $C_{g,p}^b$, perturb the map to be transverse to a point $p_i$ in that circle that is not the wedge point. Let $\gamma_i=f^{-1}(p_i)$, and note that $[\gamma_1],\dots,[\gamma_{g-p}]$ are linearly independent in $H_1\left(\Sigma,\partial \Sigma;\mathbb{Z}\right)$. Each $\gamma_i$ is a 1-submanifold of $\Sigma$, but by band-summing the components of $\gamma_i$ if necessary, we obtain a single, simple closed curve $\Gamma_i$, still in the kernel of $f_*$, and representing the same homology class. Let $\Gamma=\left\{\Gamma_1,\dots,\Gamma_{g-p}\right\}$. We then construct $W$ as above by attaching 3-dimensional 2-handles along each component of $\Gamma \times \{1\}$ to $\Sigma \times I$.
    
We claim that $\partial^+ W$ is a connected, genus $p$ surface with $b$ boundary components. First, note that $\partial^+ W$ must have $b$ boundary components since $h(w_i)=\omega_i$ for $1\leq i \leq b$ and thus none of the attaching circles $\Gamma_i$ are boundary parallel. So $\partial^0 W\cong\coprod^b S^1 \times I$. To see that $\partial^+W$ is connected, we use the fact that $[\Gamma_1], \ldots, [\Gamma_{g-p}]$ are linearly independent. Therefore no subset of $\Gamma$ can separate, for then their homology classes would sum to the trivial class. This proves the claim, thus $W=W_{g,p}^b$ is of the desired diffeomorphism type and has fundamental group isomorphic to $C_{g,p}^b$.

We now show that the inclusion map $i: \Sigma \hookrightarrow W$ induces $h$ on $\pi_1$. We can extend $f: \Sigma \to X_n$ to $\tilde{f}: W \to X_n$ since $\ker(i_*) \subset \ker(f_*)$ by construction. Thus, the composition $\tilde{f}\circ i: \Sigma \to W \to X_n$ is equal to $f: \Sigma \to X_n$. Note that $f_*=\tilde{f}_*\circ i_*$ implies that $\tilde{f}_*$ is surjective, and since $\pi_1(W)\cong \pi_1(X_n)$ are free, and therefore Hopfian, this surjection is an isomorphism. Thus, $\pi_1(W)$ is isomorphic to $C_{g,p}^b$ via $\tilde{f}_*$. Consequently, the composition 
        $$S_g^b \xrightarrow{\phi} \pi_1 (\Sigma) \xrightarrow{i_*} \pi_1 (W) \xrightarrow{\tilde{f}_*} C_{g,p}^b$$
 is equal to $h:S_g^b \to C_{g,p}^b$.
 
 Finally, this construction is unique up to diffeomorphism relative to $\partial^-W=\Sigma$ by Dehn's Lemma. 
\end{proof}

\begin{remark}\label{rmk:multipleboundary}
Although the proof presented above will always produce a relatively trisected 4-manifold with connected boundary, it is possible to allow for any finite number of boundary components by extending Lemma~\ref{lem:compression} as follows. To produce a compression body with $m > 1$ boundary components realizing a given surjection, we may take the initial map from $\Sigma$ to a graph constructed by starting with a path graph (i.e., a graph homeomorphic to an interval) with $m$ vertices and adding the appropriate number of loops to each vertex. (This is determined by the definition of $C_{g,p}^b$, which needs to be enhanced to $C_{g,\vec{p}}^{\vec{b}}$ to indicate which generators correspond to which components of $\partial^+W$ as in \cite{nickthesis}). The construction requires the additional step of removing pre-images of points in the interior of the cut edges, separating the surface $\Sigma$ into multiple components. Each of these edges corresponds to the belt sphere of a $4$--dimensional $1$--handle that connects the boundary components of the $4$--manifold, of which there are exactly $m-1$. One may then utilize the techniques found in the proof of Lemma~\ref{lem:compression} to see that each wedge of circles corresponds to a connected component of $\partial^+ W$ of the correct diffeomorphism type. In this setting, $\partial^-W = \Sigma$ and $\partial^+W$ has $m$ components each with a positive, possibly distinct, number of boundary components. 
\end{remark}

\begin{theorem}\label{thm:set_equiv_relative}
There is a map $\mathcal{M}$ from the set of relative group trisections to the set of pointed, parameterized relatively trisected 4-manifolds. Moreover, $\pi_1 \circ \mathcal M$ is the identity up to isomorphism of relatively trisected groups, and $\mathcal M \circ \pi_1$ is the identity up to diffeomorphism of 4-manifolds.
\end{theorem} 

\begin{proof}
Using Proposition~\ref{lem:compression}, we construct three such compression bodies $W_1, W_2, W_3$, one for each map $h_i:S_g^b \rightarrow C_k$, and attach $W_i\times I$ to $\Sigma_g^b \times D^2$ by identifying $W_i^- \times I$ with $\Sigma_g^b \times [t_i-\epsilon, t_i+\epsilon] \subset \Sigma_g^b \times \partial D^2$. After smoothing corners, this manifold has three boundary components which are each diffeomorphic to $W_i \cup_{\partial^{-}} W_j$. In contrast to the closed setting, we must continue constructing $\partial X_i$ for our proposed trisected manifold $X$. Let $P$ be genus $p$ surface with $b$ boundary components and $D=\{re^{i\theta}| \; 0\leq r\leq 1,\; \frac{2\pi(i-1)}{3}\leq \theta \leq \frac{2\pi i}{3}\}$ be a third of the standard disk in the plane. By Corollary 14 of \cite{cgp}, there is unique way (up to isotopy rel. boundary) to attach $(P\times I)\cup (\partial P \times D)$ to $W_i \cup_{\partial^{-}} W_j$. After doing so for each pair $(i,j)$ (and smoothing corners), we obtain a $4$--manifold with four boundary components: one given by three copies of $\cup_{(i,j)}(P\times I) \cup (\partial P \times D)$ which are attached to the compression bodies, and the other three boundary components given by $M^{ij}:=\left((P\times I)\cup (\partial P \times D)\right) \cup \left(W_i \cup_{\partial^{-}} W_j\right)$. The former is an open book decomposition, with page $P$ and binding $\partial P \times \{0\} \subset \partial P \times D^2$. Since each face of the group trisection is a pushout, each of the three diffeomorphic boundary components has fundamental group the free group of rank $k$. Note that we also make use of the requirement that $w_j\mapsto \omega_j$ in each compression body. The attached copies of $P\times I$ identifies $1$--handles between the $\partial^+ W_i$s, and each $1$--handle corresponds to an $\omega_j$.

What remains to be shown is that these three boundary components can be uniquely filled by $\natural^k S^1\times B^3$. For this we use the following argument from \cite{grouptri}. In \cite{Stallings}, Stallings showed that if the fundamental group of a $3$--manifold $Y$ surjects onto a free product $A*B$, then $Y$ can be expressed as a connect sum $Y = Y_1\#Y_2$ with $\pi_1(Y_1) = A$ and $\pi_1(Y_2)=B$. Thus, $M^{ij}=M^{ij}_1\#\cdots\# M^{ij}_k$ for each pair $(i,j).$ The $3$-dimensional Poincar\'{e} Conjecture ensures that the fundamental group of each connected summand is not trivial. However, the only closed, prime $3$--manifold with infinite cyclic fundamental group is $S^1\times S^2$. (See \cite{Hatcher3M} for details.)

Finally, we conclude the construction by filling each of the $\#^k S^1\times S^2$ boundary components uniquely with $\#^k S^1\times B^3$ by Laudenbach-Poenaru \cite{LP}.
\end{proof}

\begin{remark}\label{rmk:diagrams}
An implicit intermediate step in the above construction is the construction of a \emph{relative trisection diagram} from a $3$--splitting. The diagrammatic aspect of trisections does not come into view when discussing the categorical perspective in Section~\ref{section:category_trisection}. However, trisection diagrams are useful tools for describing explicit trisected $4$--manifolds. We recall the definition of a relative trisection diagram in Section~\ref{section:examples} where we give explicit examples of $3$--splittings and their corresponding trisected $4$--manifolds.
\end{remark}

\subsection{On maps between relative handlebodies}
\label{section:klug_results}

Understanding maps up to homotopy between handlebodies is a crucial ingredient in the categorification of trisected $4$--manifolds. We recall two results which prove useful for this purpose.

The following lemma of Klug shows that the homotopy type of maps between $4$--dimensional handlebodies is determined by their boundary maps. This immediately extends to relative trisections since each piece satisfies the hypothesis of being a $4$--dimensional handlebody.
\begin{lemma}[\hspace{1sp}{\cite[Lem. 3.2]{klugtri}}]\label{klug2} Let $W$ and $W'$ be $4$-dimensional $1$--handlebodies. Any smooth map $\partial W \to \partial W'$ extends to a map $W \to W'$. 
\end{lemma}

We also make use of the following lemma regarding the boundary maps between $4$--dimensional handlebodies.
\begin{lemma}[\hspace{1sp}{\cite[Lem. 3.3]{klugtri}}] \label{klug1} Let $X$ and $X'$ be $4$-dimensional handlebodies, and let $f, g: X \to X'$ be proper maps. Then any homotopy $\partial X \times I \to \partial X'$ between $f|_{\partial X}$ and $g|_{\partial X}$ extends to a homotopy $X\times I \to X'$ between $f$ and $g$.
\end{lemma}

In order to use the above results from \cite{klugtri}, we must have a Heegaard splitting of $\partial X_i$. However, when constructing the $4$--manifold with boundary from the trisection data, the generalized Heegaard splittings do not correspond to $\partial X_i$ as they do in the closed setting. We must first finish the construction by attaching the remaining pieces of $\partial X_i$. It is only at this point that we obtain a unique Heegaard splitting of the (closed) $3$--manifold $\partial X_i$. Lemma~\ref{lem:heegaardboundary} shows that the boundary of each codimension $0$ piece of trisection is diffeomorphic to $\# S^1\times S^2$. 

To determine the Heegaard splitting of $\#S^1\times S^2$, which is unique by Waldhausen \cite{waldhausen}, we use a standard construction which we include for completeness. Given an abstract open book $\ob = (P, \phi)$, where $P$ is a genus $p$ surface with $b$ boundary components, define $H_1  = P \times [0, 1/2]$ and $H_2 = P \times [1/2, 1]$ where we have re-parametrized $S^1$ to be the union of two intervals. Each $H_i$ is clearly a handlebody since it is a product of a surface with boundary. Given any collection $\mathcal{A}_i$ of $l=2p+b-1$ disjoint properly embedded arcs which cut $P$ into a disk, we obtain a collection of $l$ disks of which cut $H_i$ into a $3$ ball by following the arcs along the product. Attaching the final piece $\partial P \times B^2$ does not change the topology of the handlebodes, as we can decompose the $B^2$ component in to half disks $B^2 = B_1\cup B_2$ and attach $\partial P \times B_i$ to $H_i$. This can be seen as adding a collar neighborhood to the boundary of $H_i$. Thus, we have produced a genus $l$ Heegaard splitting. To see the Heegaard diagram, we note that the monodromy plays a role in attaching $H_1$ and $H_2$ along their common boundary. The Heegaard surface is $\partial H_i = P \underset{\partial}{ \cup} -P$ and the attaching map is given by $id \cup \phi$. Notice in the special case when $\phi=id$ we obtain the unique genus $l$ Heegaard splitting of $\#^l S^1\times S^2.$

\begin{lemma}\label{lem:heegaardboundary}
Given a relative $(g,k;p,b)$-trisection $X = X_1\cup X_2\cup X_3$, each boundary component $\partial X_i$ is equipped with a Heegaard splitting that is unique up to isotopy.
\end{lemma}
\begin{proof}
Let us denote by $P$ the page of the open book on $\partial X$ induced by the given trisection of $X$. For each $i$, $X_i$ can be viewed by attaching $4$--dimensional $1$--handles to a $4$--dimensional regular neighborhood of $P$. The thickened surface is $P\times B^2\cong \natural^l S^1\times B^3$, where $l = 2p+b-1$ is the number of $2$--dimensional $1$--handles of $P$. The boundary of this thickening, $(P\times \partial B^2) \cup (\partial P \times B^2),$ inherits its own open book decomposition, where $(P\times \partial B^2)$ is the mapping torus using the identity map and $(\partial P \times B^2)$ is a neighborhood of the binding . We obtain a Heegaard splitting $H_1\cup H_2$ from this open book as above. 

Attaching the $4$--dimensional $1$--handles to complete the construction of $X_i$ has the effect of boundary connect summing $S^1\times B^3$. Thus, after attaching $k$ $1$--handles, the boundary changes by a connected sum with $\#^k S^1\times S^2.$ By Waldhausen Theorem, we have a genus $k$ Heegaard splitting of $\#^k S^1\times S^2$, unique up to isotopy, which we can stabilize as needed. Let us denote this stabilized splitting as $\mathcal{H}_1 \cup \mathcal{H}_2$. Taking the connect sum of these two Heegaard splitting, we obtain the unique genus $g>l+k$ Heegaard splitting of $\partial X_i$.
\end{proof}

\section{Category of Trisections}
\label{section:category_trisection}

We provide a categorical framework to relate group and manifold trisections in the relative case (i.e., for trisected manifolds with boundary and their relation to $(g,k; p,b)$-trisections of groups), and give a functorial description of the maps $\mathcal{M}$ and $\mathcal{G}$ defined above. We begin by recalling work of Klug \cite{klugtri} in the case of closed 3-manifolds with Heegaard splitting and closed 4-manifolds with trisection, which categorifies the main result of \cite{grouptri} (see Theorem~\ref{thm:man2grp}).  We then establish a similar functorial description for relative trisections of 4-manifolds.

Recall that for categories $\mathsf C$, $\mathsf D$, a (covariant) functor $F: \mathsf C \to \mathsf D$ is an \emph{equivalence (of categories)} if it is both fully faithful and essentially surjective. This means:
\begin{itemize}
\item (faithful) for all objects $A, A'$ in $\mathsf C$, the induced map $$\text{Mor}_{\mathsf C}(A,A') \to \text{Mor}_{\mathsf D}(F(A), F(A'))$$ is injective.
\item (full) for all objects $A, A'$ in $\mathsf C$, the induced map $$\text{Mor}_{\mathsf C}(A,A') \to \text{Mor}_{\mathsf D}(F(A), F(A'))$$ is surjective.
\item (essentially surjective) each object $B$ in $\mathsf D$ is isomorphic to $F(A)$ for some object $A$ in $\mathsf C$.
\end{itemize}
We write $\mathsf C \simeq \mathsf D$ to denote the equivalence.

\subsection{Closed 3-manifolds} Using the fundamental group, one functorially encodes the data of a closed 3-manifold with Heegaard splitting in a purely algebraic way. This associates to each decomposed manifold a \emph{splitting homomorphism} \cite{klugtri}, which we call \emph{2-splitting homomorphisms}.

Let $\mathsf{Spl3Man}$ denote the category of closed 3-manifolds with Heegaard splitting and smooth maps preserving this splitting. An object of this category is given by
\begin{itemize}
\item a closed 3-manifold $M$,
\item a decomposition $M = H_1 \cup H_2$ where $H_i$ is a genus $g$ handlebody and $S: = H_1\cap H_2 = \partial H_1 = -\partial H_2$,
\item a basepoint $x \in S$,
\item parameterizations $\pi_1(H_1, x) \cong F_g \cong \pi_1(H_2, x)$, where $F_g$ is a free group of rank $g$.
\end{itemize}
Morphisms of $\mathsf{Spl3Man}$ are given by smooth maps which preserve the Heegaard splitting and basepoint, and where two maps are equivalent if they are homotopic by a homotopy which preserves the splitting and basepoint.

Moving to the algebraic setting, let $\mathsf{2SplHom}$ denote the category of pairs $$(f_1, f_2): S_g \to F_g \times F_g,$$ where each $f_i:  S_g \to F_g $ is a surjective group homomorphism. We refer to such pairs as \emph{2-splitting homomorphisms}. Taking the pushout of the homomorphisms $f_1$ and $f_2$, one may realize 2-splitting homomorphisms as diagrams 

\[
\begin{tikzcd}[row sep=scriptsize, column sep=scriptsize]
      & F_g \ar[rd] & \\
    S_g \ar[ru]  \ar[rd] & & G\\
    & F_g \ar[ur]
\end{tikzcd}
\]
which gives a group-theoretic analogue of Heegaard splittings of 3-manifolds. 

Morphisms $(f_1, f_2) \to (g_1, g_2)$ in $\mathsf{2SplHom}$ are given by triples $\varphi = (\varphi_0, \varphi_1, \varphi_2)$ of group homomorphisms $\varphi_0:S_g \to S_g$ and $\varphi_i: F_g \to F_g$ so that the following diagram commutes:
\[
\begin{tikzcd}[row sep=scriptsize, column sep=scriptsize]
      & F_g \ar[rr, "\varphi_1"]& & F_g  \\
    S_g \ar[ru, "f_1"]  \ar[rd, swap, "f_2"] \ar[rr, "\varphi_0"] &  & S_g \ar[ru, "g_1"]  \ar[rd, swap, "g_2"] & \\
    & F_g \ar[rr, swap, "\varphi_2"]& &F_g 
\end{tikzcd}
\]
Such maps determine a homomorphism between the corresponding square diagrams using the universal property of the pushout. 

Given a 3-manifold $M$ with Heegaard splitting $M = H_1 \cup H_2$, $S: = H_1\cap H_2 = \partial H_1 = \partial H_2$, applying the fundamental group to each piece of the decomposition recovers the 2-splitting homomorphism $$\pi_1(S, x) \to \pi_1(H_1, x) \times \pi_1(H_2, x),$$ where the surjective homomorphisms are induced by the inclusion maps $S \hookrightarrow H_i$. Given two 3-manifolds $M, N$ with Heegaard splittings, and a smooth map $f: M \to N$ which preserves the Heegaard splittings and basepoints (i.e., a morphism of $\mathsf{Spl3Man}$), the induced map on fundamental groups yields a 2-splitting homomorphism \cite[p. 10876]{klugtri}. This defines a functor $$\pi_1: \mathsf{Spl3Man} \to \mathsf{2SplHom}.$$ First established on the level of objects in \cite{Jaco1, Jaco2}, and then enriched to a categorical description in \cite{klugtri}, there is a functor $\mathscr M_3: \mathsf{2SplHom} \to \mathsf{Spl3Man}$ which is an equivalence.

\begin{proposition}[\hspace{1sp}\cite{klugtri}, Thm. 2.5]
The functor  $  \mathscr M_3: \mathsf{2SplHom} \to  \mathsf{Spl3Man}$ defines an equivalence of categories. 
\end{proposition}

\subsection{Closed 4-manifolds} A similar situation arises in the 4-dimensional case.  To any closed 4-manifold with trisection, we associate a \emph{3-splitting homomorphism} (cf. \cite{klugtri}). The data of a 3-splitting homomorphism is equivalent to a \emph{group trisection} as discussed in Definition~\ref{defn:group_trisection} (see also \cite{grouptri}). 

Let $\mathsf{Tris4Man}$ denote the category of closed 4-manifolds with trisection and smooth trisection-preserving maps. An object of this category is given by
\begin{itemize}
\item a closed 4-manifold $X$
\item a decomposition $X = H_1 \cup H_2 \cup H_3$
\item a base point $x \in X_1 \cap X_2 \cap X_3$
\item parameterizations $\pi_1(X_1\cap X_2) \cong \pi_1(X_2\cap X_3, x) \cong \pi_1(X_1\cap X_3, x) \cong H_g$ and $\pi_1(X_1 \cap X_2 \cap X_3, x) \cong S_g$.
\end{itemize}
Morphisms are given by smooth maps preserving the splitting and basepoint considered up to homotopies which preserve the trisection and basepoint.

Let $\mathsf{3SplitHom}$ denote the category of 3-splitting homomorphisms, i.e., triples of surjective group homomorphisms $(f_1, f_2, f_3): S_g \to H_g \times H_g \times H_g$ such that 
\begin{itemize}
\item the pushout of any pair $f_i \times f_j: S_g \to H_g \times H_g$ is isomorphic to a free group for $1 \leq i < j \leq 3$.
\end{itemize}
 We view such triples as cubical diagrams 
\begin{center}
\begin{tikzcd}[row sep=scriptsize, column sep=scriptsize]
    & H_g \arrow[r] \arrow[dr] &   Z_k \arrow[dr]  &\\
    S_g \arrow[r]\arrow[dr]\arrow[ur]  &   H_g \arrow[ur, crossing over] \arrow[dr]   &   Z_k \arrow[r]  &   G\\
    &   H_g \arrow[r] \arrow[ur] &   Z_k \arrow[ur] \arrow[ul, leftarrow, crossing over] &
\end{tikzcd}
\end{center}
by taking pushouts along each pair and then taking subsequent pushouts along the resulting maps (see Definition~\ref{defn:group_trisection}). The data of a 3-splitting homomorphism is equivalent to the data of a  group trisection. A morphism $(f_1, f_2, f_3) \to (g_1, g_2, g_3)$ of 3-splitting homomorphisms is given by a quadruple $\varphi= (\varphi_0, \varphi_1, \varphi_2, \varphi_3)$ of group homomorphisms $\varphi_0: S_g \to S_g$, $\varphi_i: H_g \to H_g$ so that the following diagram commutes:
\[
\begin{tikzcd}[row sep=scriptsize, column sep=scriptsize]
& H_g \ar[rr, "\varphi_1"]& & H_g \\
S_g \ar[rr, "\varphi_0"] \ar[ru, "f_1"] \ar[dr, "f_2"] \ar[dd, "f_3"]&  & S_g \ar[ru, "g_1"] \ar[dr, "g_2"]  \ar[dd,  near end, "g_3"']& \\
& H_g \ar[rr, near start, "\varphi_2"] & & H_g\\
 H_g \ar[rr,  "\varphi_3"] & &  H_g \ar[from=uu, crossing over] & \\
\end{tikzcd}
\]
A homomorphism in $\mathsf{3SplitHom}$ thus determines a map between the associated cubical pushout diagrams (i.e., a map between group trisections) by iteratively using the universal property of the pushout. 

To every closed trisected 4-manifold $X$, one may functorially associate a 3-splitting homomorphism (or trisected group) by taking the fundamental groups of each of $X,$ $X_i$, $X_i \cap X_j$ ($i,j = 1, 2, 3$), and $X_1 \cap X_2 \cap X_3$, along with the maps induced from the natural inclusions. In the other direction, a 3-splitting homomorphism determines a unique trisected 4-manifold (up to diffeomorphism). Morphisms of 3-splitting homomorphisms uniquely determine morphisms of the associated 4-manifolds which preserve the trisections (up to homotopy). The map $\mathcal{M}$ (see Theorem~\ref{thm:man2grp}) from the set of 3-splitting homomorphisms (or group trisections) to the set of trisected $4$-manifolds may therefore also be viewed as a functor, denoted $\mathscr M_4$.

\begin{theorem}[\hspace{1sp}\cite{klugtri}, Thm. 3.4]\label{thm:equiv_4man}
The functor $\mathscr{M}_4:  \mathsf{3SplHom} \to  \mathsf{Tris4Man}$ is an equivalence of categories.
\end{theorem}

\subsection{Trisected 4-manifolds with boundary} 
\label{subsection:cat_4man}

Let $\mathsf{Tris4Man}^{\partial}$ denote the category of relatively trisected 4-manifolds with boundary and smooth maps preserving the relative trisections. An object of this category is given by 
\begin{itemize}
\item a 4-manifold $X$ with boundary,
\item a relative trisection $X = X_1 \cup X_2 \cup X_3$, 
\item a specified basepoint $x \in X_1 \cap X_2 \cap X_3$, 
\item fixed parametrizations
\end{itemize}
\begin{center}
$\pi(X_1 \cap X_2, x ) \cong \pi_1(X_2\cap X_3, x) \cong \pi_1(X_3 \cap X_1, x) \cong C^b_{g,p},$\\
	$\pi_1(X_1 \cap X_2 \cap X_3, x) \cong S^b_g,$
\end{center}
for some non-negative integers satisfying $g-p\geq 0$ and $b>0$. 
Morphisms of $\mathsf{Tris4Man}^{\partial}$ are given by smooth maps which preserve the boundaries, the relative trisection, and basepoint, considered up to homotopies which preserves the relative trisection and basepoint.

Recall the group presentations for $S_g^b$, $C_{g,p}^b$, and $Z_k$ preceding Defintion~\ref{def:relgroup}. Let $\mathsf{3SplHom}^{\partial}$ denote the category of \emph{relative 3-splitting homomorphisms}. Objects are given by triples of surjective homomorphisms $(f_1, f_2, f_3): S^b_g \to C_{g,p}^b \times C_{g,p}^b \times C_{g,p}^b$ (for some choice of non-negative integers satisfying $g-p\geq 0$ and $b>0$) such that 
	\begin{itemize}
		\item the pushout of any pair $f_i \times f_j: S^b_g \to C_{g,p}^b \times C_{g,p}^b$ is isomorphic to a free group for $1 \leq i < j \leq 3,$
		\item for each $i, j$, we have $f_i(w_j) = \omega_j$. 
	\end{itemize}
We may view such triples as cubical diagrams (just as in Defintion~\ref{def:relgroup})
\begin{center}
\begin{tikzcd}[row sep=scriptsize, column sep=scriptsize]
    & C_{g,p}^b \arrow[r] \arrow[dr] &   Z_k \arrow[dr]  &\\
    S_g^b \arrow[r]\arrow[dr]\arrow[ur]  &   C_{g,p}^b \arrow[ur, crossing over] \arrow[dr]   &   Z_k \arrow[r]  &   G\\
    &   C_{g,p}^b \arrow[r] \arrow[ur] &   Z_k \arrow[ur] \arrow[ul, leftarrow, crossing over] &
\end{tikzcd}
\end{center}
by taking pushouts along each pair and then taking subsequent pushouts along the resulting maps. The data of a relative 3-splitting homomorphism is thus equivalent to the data of a relative group trisection. A morphism $\varphi: (f_1, f_2, f_3) \to (g_1, g_2, g_3)$ of relative 3-splitting homomorphisms is given by a quadruple $\varphi=(\varphi_0, \varphi_1, \varphi_2, \varphi_3)$ of group homomorphisms $\varphi_0: S^b_g \to S^b_g$ and $\varphi_i: C_{g,p}^b \to C_{g,p}^b$ so that the following diagram commutes:
\[
\begin{tikzcd}[row sep=scriptsize, column sep=scriptsize]
& C_{g,p}^b \ar[rr, "\varphi_1"]& & C_{g,p}^b \\
S^b_g \ar[rr, "\varphi_0"] \ar[ru, "f_1"] \ar[dr, "f_2"] \ar[dd, "f_3"]&  & S^b_g \ar[ru, "g_1"] \ar[dr, "g_2"]  \ar[dd,  near end, "g_3"']& \\
&C_{g,p}^b \ar[rr, near start, "\varphi_2"] & & C_{g,p}^b\\
C_{g,p}^b \ar[rr,  "\varphi_3"] & &  C_{g,p}^b \ar[from=uu, crossing over] & \\
\end{tikzcd}
\]
A morphism in $\mathsf{3SplitHom}^{\partial}$ thus determines a map between the pushout cubical diagrams (i.e., between the relative group trisections) by iteratively using the universal property of the pushout.

Similar to the closed 4-manifold case, to each closed trisected 4-manifold $X$ with boundary, one may functorially associate a relative 3-splitting homomorphism (or relative group trisection) by taking the fundamental groups of each of $X,$ $X_i$, $X_i \cap X_j$ ($i,j = 1, 2, 3$), and $X_1 \cap X_2 \cap X_3$, along with the maps induced from the natural inclusions. In the other direction, a relative 3-splitting homomorphism determines a unique trisected 4-manifold with boundary (up to diffeomorphism). Morphisms of relative 3-splitting homomorphisms uniquely determine morphisms of the associated 4-manifolds with boundary which preserve the relative trisections (up to homotopy). The map $\mathcal{M}$ of Theorem~\ref{thm:set_equiv_relative}, from the set of relative 3-splitting homomorphisms (or relative group trisections) to the set of trisected $4$-manifolds with boundary may therefore be viewed as a functor, denoted $\mathscr M^{\partial}_4$.

\begin{theorem}\label{thm:equiv_relative}
There is a functor $ \mathscr M_4^{\partial}: \mathsf{3SplHom}^{\partial} \to \mathsf{Tris4Man}^{\partial}$ which agrees with the map $\mathcal M$ on objects. Moreover, $\mathscr M ^{\partial}_4$ defines an equivalence of categories.
\end{theorem} 

\begin{proof}
On objects, the functor $\mathscr M_4^{\partial}$ is defined in Theorem~\ref{thm:set_equiv_relative}. More precisely, a relative 3-splitting homomorphism $f: =(f_1, f_2, f_3): S^b_g \to C_{g,p}^b \times C_{g,p}^b \times C_{g,p}^b$ defines a $(g,k; p, b)$-trisection of the group $G$ obtained by taking iterated pushouts of pairs of the $f_i$'s. Furthermore, we have $G \cong \pi_1(X^f)$ for a $(g,k; p,b)$-trisected 4-manifold $X^f$ with boundary, and we take $\mathscr M^{\partial}_4(f) = X^f$.

Given relative 3-splitting homomorphisms $f = (f_1, f_2, f_3)$ and $g = (g_1, g_2, g_3)$, and a map of relative 3-splitting homomorphisms $\varphi: (f_1, f_2, f_3) \to (g_1, g_2, g_3)$, we construct a smooth map $\mathscr{M}_4^{\partial}(\varphi): X^f \to X^g$ which preserves the boundary, trisection, and basepoint (up to homotopy). As in Lemma~\ref{lem:compression}, $\varphi$ gives rise to a map between the spines of $X^f$ and $X^g$. That is, we have a map between the triple and pairwise intersections of the codimension $0$ pieces $X^f_i$, $X_i^g$ of each trisected manifold $X^f$ and $X^g$. Though this map is not unique, we obtain a unique map between the boundaries of each piece of $X^f$ and $X^g$ via Lemma~\ref{lem:heegaardboundary}. To obtain the map $X^f \to X^g$, we then implement Lemma~\ref{klug2} and Lemma~\ref{klug1}.

Composition of morphisms is preserved under $\mathscr M_4^\partial$. Indeed, any map $f\overset{\varphi}{\rightarrow}g$ between relative 3-splittings gives a map between homomorphisms $f_i \rightarrow g_i$ for each $i$, i.e., a commutative square of group homomorphisms 
\[
\begin{tikzcd}[row sep=scriptsize, column sep=scriptsize]
S^b_g \ar[r, "\varphi_0"]  \ar[d, "f_i"]&   S^b_g  \ar[d,  "g_i"'] \\
C_{g,p}^b \ar[r,  "\varphi_i"] &   C_{g,p}^b  \\
\end{tikzcd}
\] 
The homomorphism $f_i$ yields a three-dimensional handlebody, just as in the proof of Theorem~\ref{thm:set_equiv_relative}, and likewise for $g_i$. Since $\mathscr{M}_4^{\partial}(\varphi)$ preserves the relative trisections, its restriction to $X_i^{f} \subseteq X^f$ gives a map $\mathscr M_4^\partial(\varphi_i) : X^f_i \to X^g_i$. For any pair $(i, i+1),$ we also have a map $\partial X_i^f \to \partial X_i^g$, where the domain and codomain are the codimension 0 pieces of the trisections associated to $f$ and $g$, respectively. Given composable morphisms of relative 3-splitting homomorphisms $(f_1, f_2, f_3) \overset{\varphi}{\rightarrow} (g_1, g_2, g_3) \overset{\psi}{\rightarrow} (h_1, h_2, h_3)$, we apply this argument to both $\mathscr M_4^\partial (\psi) \circ \mathscr M_4^\partial (\varphi)$ and $\mathscr M_4^\partial (\psi \circ \varphi)$, which necessarily agree on $\partial X_i^f$. They must be homotopic by Lemma~\ref{klug1}.

Lastly, we show that $\mathscr M_4^{\partial}$ is an equivalence. For essential surjectivity, note that by Theorem~\ref{thm:set_equiv_relative} any object $X$ is isomorphic to $\mathscr{M}_4^{\partial}(\pi_1 (X))$ in $\mathsf{Tris4Man}^{\partial}$.

For faithfulness: let $\varphi, \psi: (f_1, f_2, f_3) \to (g_1, g_2, g_3)$ be morphisms of relative 3-splitting homomorphisms such that $\mathscr{M}_4^\partial(\varphi) = \mathscr{M}_4^\partial(\psi)$ up to homotopy. In particular, we have that $\mathscr{M}_4^\partial(\varphi)(X^f_i)$ is diffeomorphic to $\mathscr{M}_4^\partial(\psi)(X^f_i)$ for each $i=1,2,3$. Thus, when restricting to $\partial X^f_i$, we get diffeomorphic Heegaard splittings as constructed above. Waldhausen's Theorem shows that these Heegaard splittings are in fact homotopic. Since this is now a $3$--dimensional statement, we may invoke the equivalence of categories between $\mathsf{Spl3Man}$ and $\mathsf{2SplHom}$ established in\cite{klugtri} to conclude that $\varphi_i = \psi_i$ for each $i$. It follows that $\varphi=\psi$ as desired.

For fullness: let $f = (f_1, f_2, f_3), g = (g_1, g_2, g_3)$ be relative 3-splitting homomorphisms, and let $\psi: X^f \to X^g $ be a smooth map preserving the relative trisections of the 4-manifolds with boundary $X^f=X^f_1\cup X^f_2\cup X^f_3$ and $X^g=X^g_1\cup X^g_2 \cup X^g_3$. We wish to find a morphism $\varphi: (f_1, f_2, f_3)  \to (g_1, g_2, g_3)$ in $\mathsf{3SplHom}^{\partial}$ such that $\mathscr{M}_4^{\partial}(\varphi)= \psi$. Choosing $\varphi = \pi_1(\psi)$ gives $\varphi: \pi_1(X^f_1\cup X^f_2\cup X^f_3) \rightarrow \pi_1(X^g_1\cup X^g_2 \cup X^g_3)$. Since we have quasi-invertibility on the objects of these categories, this gives two maps $\psi, \mathscr{M}_4^{\partial}(\varphi): X^f  \rightarrow X^g.$
Using the same approach as in the faithfulness argument, we obtain maps on each of the Heegaard splittings for $\partial X^f_i$ which are homotopic. Appealing to the same $3$--dimensional argument gives $\psi = \mathscr{M}_4^\partial(\varphi)$ up to homotopy, completing the proof. 
\end{proof}

\begin{remark}[Empty Boundary, see Remarks~\ref{rem:empty} and \ref{rem:grouptri}]
Relative trisections of 4-manifolds with boundary and their associated group trisections generalize the closed case, even categorically. A closed 4-manifold $X$ with $(g, k)$-trisection defines a 4-manifold with $\partial M = \emptyset$ and $(g,k; 0, 0)$-trisection. If $Y$ is another 4-manifold with $(g', k')$-trisection, a smooth map $X \to Y$ of trisected 4-manifolds which preserves the trisection also preserves the empty boundary and relative trisections. This induces a fully faithful inclusion functor $\iota: \mathsf{Tris4Man} \to \mathsf{Tris4Man}^{\partial}.$
Similarly, any $(g, k)$-trisection of a group $G$ defines a $(g,k;0,0)$-trisection of $G$, yielding a fully faithful inclusion functor $\iota _{\text{gp}}: \mathsf{3SplHom} \to \mathsf{3SplHom}^{\partial}.$
Lastly, the functors $\iota, \iota_{\text{gp}}$ commute with $\mathscr M_4$, $\mathscr M_4^{\partial}$ and $\pi_1$. 
\end{remark}

\section{Examples}
\label{section:examples}

In this section we provide explicit descriptions of group trisections and their corresponding trisected $4$--manifold via $\mathscr M_4^\partial$. One important aspect of the construction of the trisected $4$--manifold from Theorem~\ref{thm:set_equiv_relative} is the corresponding \emph{relative trisection diagram}. We remind the reader that there is a one-to-one correspondence between relative trisection diagrams and relatively trisected $4$--manifolds \cite{cgp}. We begin by recalling the definition of a relative trisection diagram.

\begin{definition}\label{def:diagram}
A $(g,k;p,b)$-trisection diagram is a $4$--tuple $(\Sigma_g^b, \alpha, \beta, \gamma)$ such that 
	\begin{itemize}
		\item $\Sigma_g^b$ is a genus $g$ surface with $b>0$ boundary components,
		\item each of $\alpha, \beta, \gamma$ is a collection of $g-p$ simple, closed, essential curves on $\Sigma_g^b$,
		\item after handleslides of the curves and diffeomorphisms of the surface, each of $(\Sigma_g^b, \alpha, \beta),$ $(\Sigma_g^b, \beta, \gamma),$ and $ (\Sigma_g^b, \alpha, \gamma)$ may be put in ``standard position,'' as shown in Figure~\ref{fig:standardpos}. 
	\end{itemize}
\end{definition}

\begin{figure}[h!]
	\includegraphics{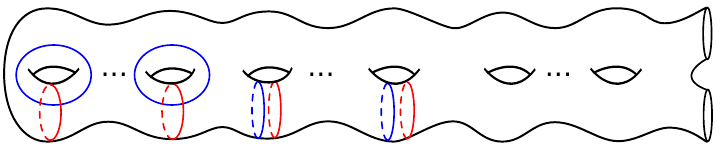}
	\caption{The standard position of a pair of multicurves in a a $(g,k;p,b)$ relative trisection diagram. \label{fig:standardpos}}
\end{figure}

\begin{example}
There is a  $(2,1;0,2)$-trisection of the trivial group. Fix the following group presentations: 
	\begin{align*}
		S_1^2 &= \left\langle x_1, y_1, w_1, w_2\left.\right|[x_1,y_1][x_2,y_2] = w_1 w_2\right\rangle\\
		C_{1,0}^2 &= \left\langle \delta_1, \omega_1, \omega_2\left.\right|1 = \omega_1 \omega_2\right\rangle.
	\end{align*}
The $3$--splitting is defined by $f_i: S_1^2 \to C_{1,0}^2$, and we may always assume that two of the maps, say $f_1$ and $f_2$, are in ``standard position,'' i.e., defined by
	\begin{itemize}
		\item$f_1(x_i) = 1$, $f_1(y_i) = \delta_i$, and $f_1(w_j) = \omega_j$
		\item$f_2(x_i) = \delta_i$, $f_2(y_i) = 1$, and $f_2(w_j) = \omega_j.$
	\end{itemize}
The group trisection is then entirely determined by the definition of the third surjective homomorphism:
	\begin{itemize}
		\item $f_3(x_1) = \delta_1,$
		\item $f_3(y_1) = \omega_1\delta_1^{-1},$
		\item $f_3(w_j) = \omega_j$, for $j=1,2$.
	\end{itemize}
Notice that the element $w_1^{-1}y_1x_1$ is in $\ker(f_3)$. This results in the relative trisection diagram of $B^4$ shown in Figure~\ref{fig:b4}.
\end{example}
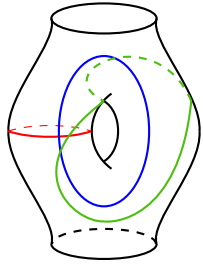
\begin{figure}[h!]\centering
\begin{tikzpicture}
	
	

	\draw [thick] (-.7,0) arc (180:360:.7 and .2);
	\draw [thick, dashed] (.69,0) arc (0:180:.69 and .2);
	\draw [thick] (0, 3) ellipse (.7cm and .2cm);
	\draw [thick] (-.7, 0) .. controls (-.6, .3) and (-1.3, 1) .. (-1.27, 1.5);
	\draw [thick] (-1.27, 1.5) .. controls (-1.3, 2) and (-.6, 2.7) .. (-.7, 3);
	\draw [thick] (.7, 0) .. controls (.6, .3) and (1.3, 1) .. (1.27, 1.5);
	\draw [thick] (1.27, 1.5) .. controls (1.3, 2) and (.6, 2.7) .. (.7, 3);
	\draw [thick] (.1, 1) .. controls (-.25, 1.25) and (-.25, 1.75).. (.1, 2);
	\draw [thick] (0,1.1) .. controls (.25, 1.25) and (.25, 1.75).. (0,1.9);
	
	\draw [thick, red] (-1.27, 1.5) .. controls (-1, 1.4) and (-.44, 1.4) .. (-.172, 1.5);
	\draw [dashed, red] (-1.27, 1.5) .. controls (-1, 1.6) and (-.44, 1.6) .. (-.172, 1.5);
	
	\draw [thick, blue] (0, 1.5) ellipse (.6cm and 1cm);
	
	\draw [thick, apple] (0, .3) .. controls (-.4,.28) and (-1.2, 1) .. (0, 1.9);
	\draw [thick,dashed, apple] (1.16, 1.9) .. controls (.7, 3) and (-.8,2.3).. (0, 1.9);
	\draw [thick, apple] (1.16, 1.9) .. controls (1, .3) and (.1, .28) .. (0, .3);
\end{tikzpicture}
\caption{A $(1,1;0,2)$-trisection diagram of $B^4$ associated to the trisected 4--manifold $\mathscr M_4^{\partial}(h_1, h_2, h_3)$.}
\label{fig:b4}
\end{figure}

\begin{example}
There is a $(2,0;0,1)$-trisection of the trivial group. Fix the group presentations from Definition~\ref{def:relgroup}:
	\begin{align*}
		S_2^1 &= \left\langle x_1, y_1, x_2, y_2, w_1\left.\right|[x_1,y_1][x_2,y_2] = w_1\right\rangle\\
		C_{2,0}^1 &= \left\langle \delta_1, \delta_2, \omega_1\left.\right|1 = \omega_1\right\rangle.
	\end{align*}
Again we assume standard position for the first two homomorphisms: 
	\begin{itemize}
		\item$f_1(x_i) = 1$, $f_1(y_i) = \delta_i$, and $f_1(w_j) = \omega_j$
		\item$f_2(x_i) = \delta_i$, $f_2(y_i) = 1$, and $f_2(w_j) = \omega_j.$
	\end{itemize}
We define our third homomorphism via
	\begin{itemize}
		\item $f_3(x_1)=\delta_1$
		\item $f_3(x_2)=\delta_2^{-1}\delta_1\delta_2$
		\item $f_3(y_1)=\delta_1$
		\item $f_3(y_2)=\delta_2$
		\item $f_3(w_1)=\omega_1$.
	\end{itemize}
It is straightforward to check that this is a relative group trisection of the trivial group. Notice that the elements $y_1w_1^{-1}y_2x_2^{-1}y_2^{-1}$ and $y_2x_1^{-1}$ are in $\ker(f_3)$. Moreover, these elements represent simple closed curves in $\Sigma_2^1$ using the identification of $\pi_1(\Sigma_2^1)$ with $S_2^1$ as in Figure~\ref{fig:Ex2Pic1}.
\begin{figure}[h!]
	\includegraphics[scale=.55]{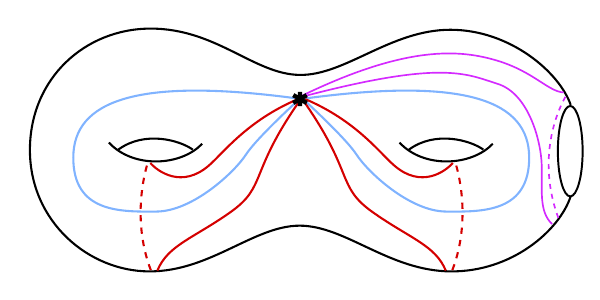}
	\caption{An identification of $S_2^1$ with $\pi_1(\Sigma_2^1)$.\label{fig:Ex2Pic1}}
\end{figure}
This results in the relative trisection diagram of $\overline{(S^2\times S^2) \backslash \{pt\}}$ given in Figure~\ref{fig:Ex2Pic2}.
\begin{figure}[h!]
	\includegraphics[scale=.55]{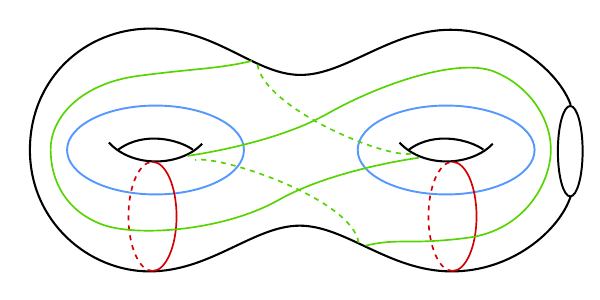}
	\caption{The corresponding relative trisection diagram corresponding to the defined $3$--splitting. We obtain the relatively trisected manifold by attaching three thickened handlebodies $H_\alpha \times I$, $H_\beta \times I$, and $H_\gamma \times I$ to $\Sigma_2^1\times D^2$ so that each pair of curves $(\alpha_1, \alpha_2), (\beta_1, \beta_2),$ and $(\gamma_1,\gamma_2)$ bound disjoint disks in the associated $H_\bullet \times \{0\}.$\label{fig:Ex2Pic2}}
\end{figure}

\begin{remark}
The above example provides a means of obtaining a $(g,k;0,1)$-trisection of a finitely presented group $G$ from a $(g,k)$-trisection of $G$. We begin by replacing the surface group $S_g = \langle x_i, y_i\left.\right| \prod[x_i,y_i] = 1\rangle$ with the free group $S_g^1$ by adding the generator $w_1$ and modifying the relation so that the product of the commutators is equal to $w_1$. The splitting homomorphisms remain the same on each $x_i$ and each $y_i$, and the definition of a relative group trisection requires that the surjections take $w_1$ to $\omega_1$. Since the original $3$-splitting $(f_1, f_2, f_3)$ yields $G$, the resulting relative $3$-splitting $(f'_1, f'_2, f'_3) \in\mathsf{3SplHom}^{\partial}$ also yields $G$, a fact which we leave to the reader to verify. Considering the associated (relatively) trisected 4-manifolds, one sees that $\mathscr M_4^{\partial}(f'_1, f'_2, f'_3)$ is diffeomorphic to $\mathscr M_4(f_1, f_2, f_3)$ with an open $4$--ball removed. Indeed, this is reflected in Figure~\ref{fig:Ex2Pic2}, which is the trisection diagram of $S^2\times S^2$ with a puncture.
\end{remark}
\end{example} 

\section{Final Remarks and Questions}
\label{section:remarks}

One of the main features of relatively trisected $4$--manifolds with boundary is the induced open book decomposition on the bounding $3$--manifold. Moreover, it is known that every open book decomposition of $\partial X$ is induced by a relative trisection of $X$. In \cite{cgp}, the authors give an explicit algorithm to obtain the open book decomposition induced by a relative trisection. The main tool in this monodromy algorithm is a collection of disjoint properly embedded arcs in the page of the open book. However, properly embedded arcs do not naturally fit in the ``apply $\pi_1$'' ethos that is necessary to understand group trisections.
\begin{question}[Translating the Monodromy Algorithm]
Is there an analog to the monodromy algorithm in the purely group theoretic setting, i.e., without factoring through $\mathsf{Tris4Man}^{\partial}$, which yields the correct boundary data when $\mathscr M_4^{\partial}$ is applied?
\end{question}

Another natural question is addressing the uniqueness of relative group trisections of a given finitely presented group. Relative trisections admit three distinct types of stabilization operations:
\begin{itemize}
	\item An interior stabilization which is equivalent to taking the connected sum with the unique genus $3$ trisection of $S^4.$ This does not affect the induced open book decomposition.
	\item A relative stabilization which is equivalent to taking the boundary connected sum with a $(1,1;0,2)$--relative trisection of $B^4$. This changes the induced open book decomposition by a positive or negative Hopf plumbing.
	\item A relative double twist which changes the spin$^c$ structure associated to the induced open book in a way that respects the structure of a trisection and preserves the $4$--manifold. (See \cite{relL} for more details.)
\end{itemize}
\begin{theorem*}[Uniqueness of Relative Trisections \cite{gk, nickthesis, relL}]
Any two relative trisections of the same $4$--manifold can be related by some number of interior stabilizations, relative stabilizations, and relative double twists.
\end{theorem*}
Of course, any uniqueness statement necessarily involves the induced open book. Thus, taking the approach that a $3$--splitting induces a $2$--splitting, we can ask for a uniqueness statement for relative group trisections:
\begin{question}[Uniqueness in $\mathsf{3SplHom}^{\partial}$]
What are the appropriate group trisection operations which give a full uniqueness statement for $3$--splittings?
\end{question}
Using the categorical equivalence of Theorem~\ref{thm:equiv_relative}, we know that such stabilizations exist since we can perform each type of stabilization topologically and then apply $\pi_1$. However, it should be possible to formulate the stabilization operation in a purely group theoretic way.

Finally, it would be interesting to incorporate the work of \cite{grouptrisurf} into the above categorical framework. Indeed, the complement of a surface $S$ in a $4$--manifold $X$ yields a $4$--manifold with (possibly disconnected) boundary which can be relatively trisected. Additionally, if the surface is in bridge position with respect to a given trisection, then by work of Miller and Kim \cite{price}, we obtain a unique relative trisection of $X\backslash S$ from the bridge trisection of $(X,S).$
\begin{question}[Subcategory of Bridge Splittings]
Can the categorical framework of Section~\ref{subsection:cat_4man} be applied to the work of Blackwell et al. \cite{grouptrisurf}? More specifically, can the methods of Miller and Kim be used to identify the set $\mathtt{Man}^{(4,2)}$ of bridge trisected surfaces in $4$--manifolds as a subcategory of $\mathsf{Tris4Man}^{\partial}$?  Does this in turn allow us to realize $\mathtt{Alg}^{(4,2)}$ as a subcategory of $\mathsf{3SplHom}^{\partial}$?
\end{question}


\bibliographystyle{alpha}
\bibliography{references}


\end{document}